\documentclass[12pt]{amsart}
\font\bbbld=msbm10 scaled\magstephalf

\newcommand{\bfH}{\hbox{\bbbld H}}
\newcommand{\bfR}{\hbox{\bbbld R}}
\newcommand{\bfS}{\hbox{\bbbld S}}
\newcommand{\ve}{{\bf e}}
\newcommand{\vn}{{\bf n}}
\newcommand{\vs}{{\bf s}}
\newcommand{\vv}{{\bf v}}
\newcommand{\vx}{{\bf x}}
\newcommand{\vz}{{\bf z}}

\newcommand{\ol}{\overline}

\newtheorem{theorem}{Theorem}[section]
\newtheorem{lemma}[theorem]{Lemma}
\newtheorem{proposition}[theorem]{Proposition}

\theoremstyle{definition}

\theoremstyle{remark}
\newtheorem{remark}[theorem]{Remark}

\numberwithin{equation}{section}

\begin{document}
\setlength{\baselineskip}{1.2\baselineskip}

\title[Hypersurfaces of Constant Curvature]
{Hypersurfaces of Constant Curvature in Hyperbolic Space I.}
\author{Bo Guan}
\address{Department of Mathematics, Ohio State University,
         Columbus, OH 43210}
\email{guan@math.osu.edu}
\author{Joel Spruck}
\address{Department of Mathematics, Johns Hopkins University,
 Baltimore, MD 21218}
\email{js@math.jhu.edu}
\author{Marek Szapiel}
\thanks{Research of the first and second authors were supported in part
by NSF grants.
}

\begin{abstract}We investigate the problem of finding complete strictly convex hypersurfaces of constant curvature in hyperbolic space with a prescribed asymptotic boundary at infinity for a general class of curvature functions.
\end{abstract}

\maketitle

\section{Introduction}
\label{gsz-I}
\setcounter{equation}{0}

In this paper we study Weingarten hypersurfaces
of constant  curvature in hyperbolic space $\bfH^{n+1}$ with a
prescribed asymptotic boundary at infinity. More precisely, given
a disjoint collection $\Gamma = \{\Gamma_1, \dots, \Gamma_m\}$ of
closed embedded $(n-1)$ dimensional submanifolds of
$\partial_\infty \bfH^{n+1}$, the ideal boundary of $\bfH^{n+1}$
at infinity, and a smooth symmetric function $f$ of $n$ variables,
we seek a complete hypersurface
$\Sigma$ in $\bfH^{n+1}$ satisfying
\begin{equation}
\label{gsz-I10}
f(\kappa[\Sigma]) = \sigma
\end{equation}
where $\kappa[\Sigma] = (\kappa_1, \dots, \kappa_n)$
denotes the hyperbolic principal curvatures of $\Sigma$ and $\sigma$ is a
constant, with the asymptotic boundary
\begin{equation}
\label{gsz-I20}
\partial \Sigma = \Gamma.
\end{equation}

We will use the half-space model,
\[ \bfH^{n+1} = \{(x, x_{n+1}) \in \bfR^{n+1}: x_{n+1} > 0\} \]
equipped with the hyperbolic metric
\begin{equation}
\label{hb-I5}
 ds^2 = \frac{\sum_{i=1}^{n+1}dx_i^2}{x_{n+1}^2}.
\end{equation}
Thus $\partial_\infty \bfH^{n+1}$ is naturally identified with
$\bfR^n = \bfR^n \times \{0\} \subset \bfR^{n+1}$ and (\ref{gsz-I20}) may
be understood in the Euclidean sense. For convenience we say $\Sigma$ has
compact asymptotic boundary if
$\partial \Sigma \subset \partial_\infty \bfH^{n+1}$ is compact with respect
the Euclidean metric in $\bfR^n$.

The function $f$ is assumed to satisfy the fundamental structure
conditions:
\begin{equation}
\label{gs5-I20}
f_i (\lambda) \equiv \frac{\partial f (\lambda)}{\partial \lambda_i} > 0
  \;\; \mbox{in $K$}, \;\; 1 \leq i \leq n,
\end{equation}
\begin{equation}
\label{gs5-I30}
\mbox{$f$ is a concave function in $K$},
\end{equation}
and
\begin{equation}
\label{3I-36}
 f > 0 \;\;\mbox{in $K$},
  \;\; f = 0 \;\;\mbox{on $\partial K$}
\end{equation}
where $K \subset \bfR^n$ is an open symmetric convex cone such that
\begin{equation}
\label{gsz-I35}
 K^+_n := \big\{\lambda \in \bfR^n:
   \mbox{each component $\lambda_i > 0$}\big\} \subset K.
 \end{equation}
 In addition, we shall assume that $f$ is normalized
\begin{equation}
\label{gsz-I45}
f(1, \dots, 1) = 1
\end{equation}
and satisfies the following more technical assumptions
\begin{equation}
\label{gsz-I40}
\mbox{ $f$ is homogeneous of degree one}
\end{equation}
and
\begin{equation}
\label{gs5-I45}
\lim_{R \rightarrow + \infty}
   f (\lambda_1, \cdots, \lambda_{n-1}, \lambda_n + R)
    \geq 1 + \varepsilon_0 \;\;\;
\mbox{uniformly in $B_{\delta_0} ({\bf 1})$}
\end{equation}
for some fixed $\varepsilon_0 > 0$ and $\delta_0 > 0$,
where $B_{\delta_0} ({\bf 1})$ is the ball
of radius $\delta_0$ centered at ${\bf 1} = (1, \dots, 1) \in \bfR^n$.

All these assumptions are satisfied by $f = 
(\sigma_k/\sigma_l)^{\frac{1}{k-l}}$, $0 \leq l < k \leq n$, defined
in $K_k$ where $\sigma_k$ is the normalized $k$-th elementary
symmetric polynomial ($\sigma_0 =1$) and
\[ K_k = \{\lambda \in \bfR^n: \sigma_j (\lambda) > 0,
\; \forall \; 1 \leq j \leq k\}. \]
See \cite{CNS3} for proof of (\ref{gs5-I20}) and (\ref{gs5-I30}).
For  (\ref{gs5-I45}) one easily computes that
\[\lim_{R \rightarrow + \infty}
   f (\lambda_1, \cdots, \lambda_{n-1}, \lambda_n + R)
  = \Big(\frac{k}l\Big)^{\frac1{k-l}}~.\]
  Since $f$ is symmetric, by (\ref{gs5-I30}),
(\ref{gsz-I45}) and (\ref{gsz-I40}) we have
\begin{equation}
\label{gsz-I200}
f (\lambda) \leq f ({\bf 1}) + \sum f_i ({\bf 1}) (\lambda_i - 1)
= \sum f_i ({\bf 1}) \lambda_i  = \frac{1}{n} \sum \lambda_i
\;\;\mbox{in $K \subset K_1$}
\end{equation}
and
\begin{equation}
\label{gsz-I210}
 \sum f_i (\lambda) = f (\lambda) + \sum f_i (\lambda) (1 - \lambda_i)
\geq f ({\bf 1}) = 1 \;\;\mbox{in $K$}.
\end{equation}
Moreover, (\ref{gs5-I20}) and $f (0) = 0$ imply that
\begin{equation}
\label{gs5-I40'}
f > 0 \;\; \mbox{in $K^+_n$}.
\end{equation}

In this paper we shall focus on the case of finding complete hypersurfaces
satisfying (\ref{gsz-I10})-(\ref{gsz-I20}) with positive hyperbolic
principal curvatures everywhere; for convenience we shall call such
 hypersurfaces {\em (hyperbolically) locally strictly convex}.
In Part II \cite{GS08} we will allow
 f satisfying (\ref{gs5-I20})-(\ref{gs5-I45})
and general cones K.

Before we state our first result we need to explain the orientation
of hypersurfaces under consideration. In this paper all
hypersurfaces in $\bfH^{n+1}$ we consider are assumed to be
connected and orientable. If $\Sigma$ is a complete hypersurface in
$\bfH^{n+1}$ with compact asymptotic boundary at infinity, then the
normal vector field of $\Sigma$ is chosen to be the one pointing to
the unique unbounded region in $\bfR^{n+1}_+ \setminus \Sigma$, and
the (both hyperbolic and Euclidean) principal curvatures of $\Sigma$
are calculated with respect to this normal vector field.

\begin{theorem}
\label{gsz-th10}
Let $\Sigma$ be a complete locally strictly convex $C^2$ hypersurface in
$\bfH^{n+1}$ with compact asymptotic boundary at infinity.
Then $\Sigma$ is the (vertical) graph of a function
$u \in C^2 (\Omega) \cap C^0 (\ol{\Omega})$, $u > 0$ in $\Omega$
and  $u = 0$ on $\ol{\Omega}$, for some domain $\Omega \subset \bfR^n$:
\[ \Sigma = \big\{(x, u (x)) \in \bfR^{n+1}_+: x \in \Omega\big\} \]
such that
\begin{equation}
\label{gsz-I50}
 \{\delta_{ij} + u_i u_j + u u_{ij}\} > 0 \;\; \mbox{in $\Omega$.}
\end{equation}
That is, the function $u^2 + |x|^2$ is strictly convex.
Moreover,
\begin{equation}
\label{gsz-I60}
e^u \sqrt{1 + |Du|^2} \leq
 \max \Big\{\max_{\ol{\Omega}} e^u,
\max_{\partial \Omega} \sqrt{1 + |Du|^2} \Big\}  \;\; \mbox{in $\Omega$}.
\end{equation}
\end{theorem}

According to Theorem~\ref{gsz-th10}, Problem (\ref{gsz-I10})-(\ref{gsz-I20})
for complete locally strictly convex hypersurfaces reduces to the
Dirichlet problem for a fully nonlinear second order equation which we
shall write in the form
\begin{equation}
\label{gsz-I70}
G(D^2u, Du, u) = \frac{\sigma}{u},
\;\; u > 0 \;\;\; \text{in $\Omega \subset \bfR^n$}
\end{equation}
with the boundary condition
\begin{equation}
\label{gsz-I80}
             u = 0 \;\;\;    \text{on $\partial \Omega$}.
\end{equation}
In particular, $\Gamma$ must be the boundary of some bounded domain
$\Omega$ in $\bfR^n$. The exact formula of $G$ will be given in
Section~\ref{gsz-F}.

We seek solutions of equation (\ref{gsz-I70}) satisfying
(\ref{gsz-I50}). Following the literature we call such solutions
{\em admissible}. By \cite{CNS3} condition~(\ref{gs5-I20}) implies
that equation (\ref{gsz-I70}) is elliptic for admissible solutions.
Our goal is to show that the Dirichlet problem
(\ref{gsz-I70})-(\ref{gsz-I80}) admits smooth admissible solutions
for all $0 < \sigma < 1$ which is also a necessary condition by the
comparison principle under conditions (\ref{gsz-I45}) and
(\ref{gsz-I40}), as we shall see in Section~\ref{gsz-L}.
Due to the special nature of the problem, there are substantial technical
difficulties to overcome and we have not yet succeeded in finding solutions
for all $\sigma \in (0,1)$. However we shall prove

\begin{theorem}
\label{gsz-th20}
Let $\Gamma = \partial \Omega \times \{0\} \subset \bfR^{n+1}$
where $\Omega$ is a bounded smooth domain in $\bfR^n$.
Suppose that $\sigma \in (0, 1)$ satisfies
$\sigma^2 > \frac18$.
Under conditions (\ref{gs5-I20})-(\ref{gs5-I45}) with $K = K_n^+$,
there exists a complete locally strictly convex hypersurface
$\Sigma$ in $\bfH^{n+1}$ satisfying (\ref{gsz-I10})-(\ref{gsz-I20})
with uniformly bounded principal curvatures
\begin{equation}
\label{gsz-I100}
 |\kappa[\Sigma]| \leq C \;\; \mbox{on $\Sigma$}.
\end{equation}
Moreover, $\Sigma$ is the graph of an admissible solution $u \in C^\infty
(\Omega) \cap C^1 (\bar{\Omega})$ of the Dirichlet problem
(\ref{gsz-I70})-(\ref{gsz-I80}). Furthermore, $u^2 \in C^{\infty} (\Omega)\cap C^{1,1}(\ol{\Omega}) $
and
\begin{equation}
\label{hb-I20}
\begin{aligned}
&\,\sqrt{1 + |Du|^2} \leq \frac{1}{\sigma}, \;\; u|D^2 u| \leq C
\;\;\; \mbox{in $\Omega$}\\
&\, \sqrt{1 + |Du|^2} = \frac{1}{\sigma}
\;\;\; \mbox{on $\partial \Omega$}
\end{aligned}
\end{equation}
\end{theorem}
For Gauss curvature, $f (\lambda)=(\Pi \lambda_i)^{\frac1n}$,  Theorem \ref{gsz-th20} was proved by Rosenberg and Spruck \cite{RS94} who in fact allowed all $\sigma \in (0,1)$.\\


As we  shall see in Section~\ref{gsz-F}, equation (\ref{gsz-I70}) is
singular where $u = 0$. It is therefore natural to approximate the
boundary condition (\ref{gsz-I80}) by
\begin{equation}
\label{gsz-I80'}
             u = \epsilon > 0 \;\;\;  \text{on $\partial \Omega$}.
\end{equation}
When $\epsilon$ is sufficiently small, the Dirichlet problem
(\ref{gsz-I70}),(\ref{gsz-I80'}) is solvable for all $\sigma \in (0, 1)$.

\begin{theorem}
\label{gsz-th30}
Let $\Omega$ be a bounded smooth domain in $\bfR^n$
and $\sigma \in (0, 1)$. Suppose $f$ satisfies
(\ref{gs5-I20})-(\ref{gs5-I45}) with $K = K_n^+$. Then for any
$\epsilon > 0$ sufficiently small, there exists an admissible
solution $u^{\epsilon} \in C^\infty (\bar{\Omega})$ of the Dirichlet
problem (\ref{gsz-I70}),(\ref{gsz-I80'}). Moreover, $u^{\epsilon}$
satisfies the {\em a priori} estimates
\begin{equation}
\label{hb-I15}
\sqrt{1 + |D u^{\epsilon}|^2} \leq \frac{1}{\sigma} + \epsilon C
\;\;\; \mbox{in $\Omega$}
\end{equation}
\begin{equation}
\label{hb-I25}
u^{\epsilon}|D^2 u^{\epsilon}| \leq \frac{C}{\epsilon^2}
\;\;\; \mbox{in $\Omega$}
\end{equation}
where $C$ is independent of $\epsilon$.
\end{theorem}

 The organization of the paper is as follows. Section 2 summarizes the basic information
 about vertical and radial graphs that we will need in the sequel.  In section 3 we prove global gradient bounds and some sharp estimates on the vertical component of the upward normal near the boundary. These are essential 
 for the boundary second derivative estimates which we derive in Section 4. Here we make essential
 use of the exact form of the linearized operator to derive the mixed normal-tangential estimates and assumption \eqref{gs5-I45} for the pure normal second derivative estimate. In Section 5 we prove
 a maximum principle for the maxmum hyperbolic principle curvature. Our approach uses radial graphs
 and is new and rather delicate. It is here that we have had to restrict the allowable $\sigma\in (0,1)$
 to $\sigma^2>\frac18$. Otherwise our approach is completely general and we expect Theorem
 \ref{gsz-th20} is valid for all  $\sigma \in (0,1)$. Because the linearized operator is not necessarily
 elliptic we prove Theorem \ref{gsz-th20} by an iterative procedure which is carried out in Section 6. 
 Because of this, we have derived all our estimates for a fairly general class of hypersurfaces of
 prescribed curvature as a function of position.

\bigskip

\section{Formulas for hyperbolic principal curvatures}
\label{gsz-F}
\setcounter{equation}{0}

Let $\Sigma$ be a hypersurface in $\bfH^{n+1}$ and $g$ the induced metric
on $\Sigma$ from $\bfH^{n+1}$. For convenience we call $g$ the hyperbolic
metric, while the Euclidean metric on $\Sigma$ means the induced metric
from $\bfR^{n+1}$.
We use $X$ and $u$ to denote the position vector and {\em height function},
defined as
\[ u = X \cdot \ve, \]
of $\Sigma$ in $\bfR^{n+1}$, respectively. Here and throughout this
paper, $\ve$ is the unit vector in the positive $x_{n+1}$ direction
in $\bfR^{n+1}$, and `$\cdot$' denotes the Euclidean inner product
in $\bfR^{n+1}$. We assume $\Sigma$ is orientable and let $\vn$ be a
fixed global unit normal vector field to $\Sigma$ with respect to
the hyperbolic metric. This also determines an Euclidean unit normal
$\nu$ to $\Sigma$ by the relation
\[ \nu = \frac{{\bf n}}{u}. \]
We denote $\nu^{n+1} = \ve \cdot \nu$.

Let $\ve_1, \cdots, \ve_n$ be a local orthnormal frame of vector
fields on $(\Sigma, g)$. The second fundamental form of $\Sigma$ is
locally given by
\[ h_{ij} := \langle \nabla_{\ve_i} \ve_j, {\bf n} \rangle \]
where $ \langle \cdot, \cdot \rangle$ and $\nabla$ denote the metric
and Levi-Civita connection of $\bfH^{n+1}$ respectively. The
(hyperbolic) principal curvatures of $\Sigma$, denoted as $\kappa
[\Sigma] = (\kappa_1, \dots, \kappa_n)$, are the eigenvalues of the
second fundamental form. The relation between $\kappa
[\Sigma]$ and the Euclidean principal curvatures $\kappa^E [\Sigma]
= (\kappa_1^E, \dots, \kappa_n^E)$ is given by
\begin{equation}
\label{gsz-F10}
\kappa_i = u \kappa_i^E + \nu^{n+1} \;\;\; 1 \leq i \leq n.
\end{equation}

We shall derive equations for $\Sigma$ based on this formula when
$\Sigma$ satisfies (\ref{gsz-I10}).

\medskip
\subsection{Vertical graphs}
\label{gsz-FV}

Suppose $\Sigma$ is locally represented as the graph of a function
$u \in C^2 (\Omega)$, $u > 0$, in a domain $\Omega \subset \bfR^n$:
\[ \Sigma = \{(x, u (x)) \in \bfR^{n+1}: \; x \in \Omega\}. \]
In this case we take $\nu$ to be the upward (Euclidean) unit normal vector
field to $\Sigma$:
\[ \nu = \Big(\frac{-Du}{w}, \frac{1}{w}\Big), \;\; w=\sqrt{1+|Du|^2}. \]
The Euclidean metric 
and second fundamental form of $\Sigma$ are given respectively by
\[ g^E_{ij} = \delta_{ij} + u_i u_j, \]
and
\[ h^E_{ij} = \frac{u_{ij}}{w}. \]
According to \cite{CNS4}, the Euclidean principal curvatures
$\kappa^E [\Sigma]$ are the eigenvalues of the symmetric matrix
$A^E [u] = [a^E_{ij}]$:
\begin{equation}
\label{gs5-D15}
 a^E_{ij} := \frac{1}{w} \gamma^{ik} u_{kl} \gamma^{lj},
\end{equation}
where
\[ \gamma^{ij} = \delta_{ij} - \frac{u_i u_j}{w (1 + w)}. \]
Note that the matrix $\{\gamma^{ij}\}$ is invertible with inverse
\[ \gamma_{ij} = \delta_{ij} + \frac{u_i u_j}{1 + w} \]
which is the square root of $\{g^E_{ij}\}$, i.e., $\gamma_{ik}
\gamma_{kj} = g^E_{ij}$.
>From \eqref{gsz-F10} we see that the hyperbolic principal curvatures
$\kappa [u]$ of $\Sigma$ are the eigenvalues of the matrix
$A^{\vv} [u] = \{a^{\vv}_{ij} [u]\}$:
\begin{equation}
\label{gsz-F20}
 a^{\vv}_{ij} [u] :=
\frac{1}{w} \Big(\delta_{ij}+ u\gamma^{ik} u_{kl} \gamma^{lj}\Big).
\end{equation}

Let $\mathcal{S}$ be the vector space of $n \times n$ symmetric matrices
and
\[ \mathcal{S}_K = \{A \in \mathcal{S}: \lambda (A) \in K\}, \]
where $\lambda (A) = (\lambda_1, \dots, \lambda_n)$ denotes the
eigenvalues of $A$. Define a function $F$ by
\begin{equation}
\label{gsz-F30}
F (A) = f (\lambda (A)), \;\; A \in \mathcal{S}_K.
\end{equation}

We recall some properties of $F$.
Throughout the paper we denote 
\begin{equation}
\label{gsz-F40}
F^{ij} (A) = \frac{\partial F}{\partial a_{ij}} (A), \;\;
  F^{ij, kl} (A) = \frac{\partial^2 F}{\partial a_{ij} \partial a_{kl}} (A).
\end{equation}
The matrix $\{F^{ij} (A)\}$, which is symmetric, has eigenvalues
$f_1, \ldots, f_n$, and therefore  is positive definite for
$A \in \mathcal{S}_K$ if $f$ satisfies (\ref{gs5-I20}),
 while (\ref{gs5-I30}) implies that $F$ is concave for
$A \in \mathcal{S}_K$ (see \cite{CNS3}), that is
\begin{equation}
\label{gsz-F50}
 F^{ij, kl} (A) \xi_{ij} \xi_{kl} \leq 0,
     \;\; \forall \; \{\xi_{ij}\} \in \mathcal{S}, \; A \in \mathcal{S}_K.
\end{equation}
We have 
\begin{equation}
\label{gs5-P30}
 F^{ij} (A) a_{ij} = \sum f_i (\lambda (A)) \lambda_i,
\end{equation}
\begin{equation}
\label{gs5-P40}
F^{ij} (A) a_{ik} a_{jk} = \sum f_i (\lambda (A)) \lambda_i^2.
\end{equation}

Finally, the function $G$ in equation (\ref{gsz-I70}) is determined by
\begin{equation}
\label{gs5-F60}
G (D^2 u, Du, u) = \frac{1}{u} F (A^{\vv} [u]).
\end{equation}
where $A^{\vv} [u] = \{a^{\vv}_{ij} [u]\}$ is given by (\ref{gsz-F20}).
Note that
\begin{equation}
\label{gs5-F70}
\begin{aligned}
\,&G^{st}[u] := \frac{\partial G}{\partial u_{st}}=\frac{1}{w} F^{ij}\gamma^{is}\gamma^{tj}\\
\,&G^{st}[u] u_{st}=\frac1u \Big(F^{ij}a_{ij}-\frac1w\sum F^{ii}\Big)\\
\,&G_u=-\frac1{u^2 w}\sum F^{ii}
\end{aligned}
\end{equation}
and
\begin{equation}
\label{gs5-F80}
G^{pq, st} [u]:= \frac{\partial^2 G}{\partial u_{pq} \partial u_{st}}
=\frac{u}{w^2}  F^{ij,kl} \gamma^{is}\gamma^{tj} \gamma^{kp}\gamma^{ql}
\end{equation}
where $F^{ij} = F^{ij} (A^{\vv}[u])$, etc.
It follows that, under condition (\ref{gs5-I20}),
equation~(\ref{gsz-I70}) is elliptic
 for $u$ if  $A^{\vv}[u] \in \mathcal{S}_K$, while (\ref{gs5-I30})
implies that $G(D^2 u, Du, u)$ is concave with
respect to $D^2 u$.

For later use in section \ref{gsz-P}  note that if  u is a solution of
\begin{equation}
\label{gsz-175}
\tilde{G}(D^2 u, Du, u)=G(D^2 u, Du, u)-\psi(x,u)=0
\end{equation}
then from (\ref{gs5-F70})
\begin{equation}
\label{gsz-180}
\tilde{G}_u=\frac1{u^2} \Big(\Psi-u\Psi_u-\frac1w \sum f_i\Big)
\end{equation}
where $\Psi(x,u)=u\psi(x,u)$. Since $\sum f_i \geq 1$, we obtain from (\ref{gsz-180})
\begin{equation}
\label{gsz-185}
\tilde{G}_u  \leq \frac1{u^2} \Big(\Psi-u\Psi_u-\frac1w\Big).
\end{equation}


\medskip

\subsection{Radial graphs}
\label{gsz-FR}

Let $\nabla$ denote the covariant derivative on the standard unit
sphere $\bfS^n$ in $\bfR^{n+1}$ and $y = \ve \cdot {\bf z}$ for
${\bf z} \in \bfS^n \subset \bfR^{n+1}$. Let $\tau_1, \cdots,
\tau_n$ be an orthnormal local frame of smooth vector
fields on the upper hemisphere $\bfS^n$ so that 
$\tau_i \cdot \tau_j = \delta_{ij} $.
For a function $v$ on $\bfS^n$,
we denote $v_i = \nabla_i v = \nabla_{\tau_i} v$,
$v_{ij} = \nabla_j \nabla_i v$, etc.

Suppose that locally $\Sigma$ is a radial graph over the upper hemisphere
$\bfS^n_+ \subset \bfR^{n+1}$, i.e., it is locally represented as
\begin{equation}
\label{gsz-F200}
X = e^v {\bf z}, \;\;\; {\bf z} \in \bfS^n_+ \subset \bfR^{n+1}.
\end{equation}
The Euclidean metric, outward unit normal vector and second
fundamental from of $\Sigma$ are
\[ g^E_{ij} = e^{2v} (\delta_{ij} + v_i  v_j) \]
\[ \nu = \frac{{\bf z} - \nabla v}{w}, \;\;\; w = (1 + |\nabla v|^2)^{1/2} \]
and
\[ h^E_{ij} = \frac{1}{w} e^v (v_{ij} - v_i v_j - \delta_{ij}) \]
respectively. Therefore the Euclidean principal curvatures are the
eigenvalues of the matrix
\begin{equation}
 a^E_{ij} = \frac{1}{w} e^{-v}
    \gamma^{ik}(v_{kl} - v_k v_l - \delta_{kl})  \gamma^{lj}
    = \frac{1}{w} e^{-v} (\gamma^{ik} v_{kl} \gamma^{lj} - \delta_{ij})
\end{equation}
where
\[ \gamma^{ij} = \delta_{ij} - \frac{v_i v_j}{w (1 + w)}. \]
Note that the height function $u = y e^v$. We see that the
hyperbolic principal curvatures are the eigenvalues of the matrix
$A^{\vs} [v] = \{a^{\vs}_{ij}[v]\}$:
\begin{equation}
\label{gsz-F20'}
 a^{\vs}_{ij}[v] := \frac{1}{w} (y \gamma^{ik} v_{kl} \gamma^{lj}
               - \ve \cdot \nabla v \delta_{ij}).
\end{equation}
In this case equation (\ref{gsz-I10}) takes the form
\begin{equation}
\label{gsz-F100}
F (A^{\vs} [v]) = \sigma.
\end{equation}
\bigskip

\section{Locally convex hypersurfaces. $C^1$ estimates}
\label{gsz-L}

\medskip

\subsection{Equidistance spheres}

There are two important facts which will be used repeatedly. One is
the invariance of equation (\ref{gsz-I10}) under scaling $X \mapsto
\lambda X$ in $\bfR^{n+1}$, as it is an isometry of $\bfH^{n+1}$.
The other is that the Euclidean spheres, known as equidistance
spheres, have constant hyperbolic principal curvatures. Let $B_R
(a)$ be a ball of radius $R$ centered at $a = (a', -\sigma R) \in
\bfR^{n+1}$ where $\sigma \in (0,1)$ and
 $S = \partial B_R (a) \cap \bfH^{n+1}$. Then $\kappa_i [S] = \sigma$ for all
$1 \leq i \leq n$ with respect to its outward normal. These spheres
may serve as barriers in many situations. Especially, we have the
following estimates which were first derived in \cite{GS00} for
hypersurfaces of constant mean curvature.

\begin{lemma}
\label{hb-lemma-L10}
Suppose $f$ satisfies (\ref{gs5-I20}), (\ref{gsz-I45}) and (\ref{gsz-I40}).
Let $\Sigma$ be a hypersurface in $\bfH^{n+1}$
with $\kappa [\Sigma] \in K$
and
\[ \sigma_1 \leq f (\kappa [\Sigma]) \leq \sigma_2 \]
where $0 \leq \sigma_1 \leq \sigma_2 \leq 1$ are constants,
and $\partial \Sigma \subset P (\epsilon) \equiv \{x_{n+1} = \epsilon\}$,
$\epsilon \geq 0$. Let $\Omega$ be the region in $\bfR^n$ bounded by
the projection of $\partial \Sigma$ to $\bfR^n = \{x_{n+1} = 0\}$
(such that $\bfR^n \setminus \Omega$ contains an unbounded component), and
$u$ denote the height function of $\Sigma$.

(i) For any point $(x, u) \in \Sigma$,
\begin{equation}
\label{hb-P105}
 \frac{\epsilon \sigma_2}{1+\sigma_2} + \mbox{d} (x) \sqrt{\frac{1-\sigma_2}{1+\sigma_2}}
  \leq u \leq
  \frac{L}{2} \sqrt{\frac{1-\sigma_1}{1+\sigma_1}} + \epsilon,
\end{equation}
where $d (x)$ and $L$ denote the distance from
$x \in \bfR^n$ to $\partial \Omega$ and
the (Euclidean) diameter of $\Omega$,  respectively.

(ii) Assume that $\partial \Sigma \in C^2$.
For $\epsilon > 0$ sufficiently small,
\begin{equation}
\label{hb-P110}
 \sigma_1 - \frac{\epsilon \sqrt{1-\sigma_1^2}}{r_1}
   - \frac{\epsilon^2 (1+\sigma_1)}{r_1^2}
   < \nu^{n+1}
        < \sigma_2 + \frac{\epsilon \sqrt{1-\sigma_2^2}}{r_2} +
        \frac{\epsilon^2 (1-\sigma_2)}{r_2^2}
\;\;\; \mbox{on $\partial \Sigma$}
\end{equation}
where $r_1$ and $r_2$ are
the maximal radii of exterior and interior spheres to
$\partial \Omega$, respectively.
In particular, if $\sigma_1 = \sigma_2 = \sigma$ then $\nu^{n+1} \rightarrow \sigma$
on $\partial \Sigma$ as $\epsilon \rightarrow 0$.
\end{lemma}

While Lemma~\ref{hb-lemma-L10} was proved in \cite{GS00} only for the
mean curvature case, the proof remains valid for more general symmetric
functions of principal curvatures with minor modifications.
So we omit the proof here.

Another important class of hypersurfaces of constant principal
curvatures are the horospheres $P (\epsilon) \equiv \{x_{n+1} =
\epsilon\}$, $\epsilon > 0$. Indeed, from (\ref{gsz-F10}) we see
that $\kappa [P (\epsilon)] = {\bf 1}$. By the comparison principle
we immediately obtain the following necessary condition for the
solvability of problem (\ref{gsz-I10})-(\ref{gsz-I20}).

\begin{lemma}
\label{hb-lemma-L20}
Suppose that $f$ satisfies (\ref{gs5-I20}) and (\ref{gsz-I45}),
and that there is a hypersurface $\Sigma$ in $\bfH^{n+1}$
which satisfies (\ref{gsz-I10}) and (\ref{gsz-I20}) with
$\kappa [\Sigma] \in K$. Then $\sigma < 1$.
\end{lemma}


\medskip

\subsection{Locally strictly convex hypersurfaces}
We now consider hypersurfaces of positive principal curvatures
in $\bfH^{n+1}$; we call such hypersurfaces {\em locally strictly convex}.

\begin{lemma}
\label{hb-lemma-L30}
Let $\Sigma \subset \{x_{n+1} \geq c\}$ be a locally strictly convex
hypersurface of class $C^2$ in $\bfH^{n+1}$ with compact
(asymptotic) boundary
 $\partial \Sigma \subset \{x_{n+1} = c\}$ for some constant
$c \geq 0$.
Then $\Sigma$ is a vertical graph. In particular, $\partial \Sigma$ is
is boundary of a bounded domain in $\{x_{n+1} = c\}$.
\end{lemma}

\begin{proof}
Let $T$ be the set of $t \geq c$ such that $\Sigma_t := \Sigma \cap
\{x_{n+1} \geq t\}$ is a vertical graph and let $t_0$ be the minimum
of $T$ which is clearly nonempty. Suppose $t_0 > c$. Then there must
be a point $p \in \partial \Sigma_{t_0}$ with $\nu^{n+1} (p) = 0$,
that is, the normal vector to $\Sigma$ at $p$ is horizontal. It
follows from (\ref{gsz-F10}) that $\kappa_i^E = \kappa_i/t_0 > 0$
for all $1 \leq i \leq n$ at $p$. On the other hand, the curve
$\Sigma \cap P$ (near $p$) clearly has nonpositive curvature at $p$
(with respect to the normal $\nu (p)$), where $P$ is the plane
through $p$ spanned by $\ve$ and $\nu (p)$. This is a contradiction,
proving that $t_0 = c$.
\end{proof}

By the formula (\ref{gsz-F20}) the graph of a function $u$ is
locally strictly convex if and only if the function
$U = |x|^2 + u^2$
is (locally) strictly convex, i.e., its Hessian $D^2 U$ is positive definite.
We define the class of {\em admissible} functions in a domain
$\Omega \subset \bfR^n$ as
\begin{equation}
\label{gsz-E30}
\mathcal{A} (\Omega) =
 \left\{u\in C^2 (\Omega): u > 0, \text
{$|x|^2 + u^2$ is locally strictly convex in $\Omega$}\right\}.
\end{equation}
By the convexity of $|x|^2 + u^2$ we immediately have
\begin{equation}
\label{gsz-L130}
|Du| \leq \frac{1}{u} \left(L + \max_{\partial \Omega} u |Du|\right)
\end{equation}
where $L$ is the diameter of $\Omega$. The following gradient estimate,
which improves (\ref{gsz-L130}) in the sense that it is independent of the
(positive) lower bound of $u$, will be crucial to our results.

\begin{lemma}
\label{hb-lemma-L40}
Let $u \in \mathcal{A} (\Omega)$. Then
\begin{equation}
\label{gsz-L140}
 e^u \sqrt{1 + |Du|^2} \leq  \max \Big\{\sup_{\Omega} e^u,
     \max_{\partial \Omega} e^u \sqrt{1 + |Du|^2}\Big\}.
\end{equation}
\end{lemma}

\begin{proof}
If $e^u \sqrt{1 + |Du|^2}$ attains its maximum at an
interior point $x_0 \in \Omega$ then at $x_0$,
\[  \sum_j u_j (\delta_{ij} + u_i u_j + u_{ij})
    = e^{-u} \frac{\partial}{\partial x_i} (e^u \sqrt{1 + |Du|^2}) = 0 , \;\;
   \forall \; 1 \leq i \leq n. \]
If follows that $D u (x_0) = 0$ as the matrix $\{\delta_{ij} +
u_i u_j + u_{ij}\}$ is positive definite.
\end{proof}

\begin{lemma}
\label{hb-lemma-L50}
 Let $u \in \mathcal{A} (\Omega)$ satisfy
\begin{equation}
\label{gsz-L150}
\begin{cases}
\sigma_1 \leq f (\kappa [u]) \leq \sigma_2, & \;\; \mbox{in $\Omega$} \\
\;\;\;\;\;\; u \geq \epsilon, & \;\; \mbox{in $\Omega$} \\
\;\;\;\;\;\; u = \epsilon, & \;\;  \mbox{on $\partial \Omega$}
\end{cases}
\end{equation}
where $0 < \sigma_1 \leq \sigma_2 < 1$,  $\epsilon \geq 0$ and $\partial
\Omega \in C^2$. Suppose $f$ satisfies (\ref{gs5-I20}),
(\ref{gsz-I45}) and (\ref{gsz-I40}). Then, for $\epsilon$
sufficiently small,
\begin{equation}
\label{gsz-L160}
 \frac{1}{\sqrt{1 + |Du|^2}} \geq \sigma_1 - C \epsilon (\epsilon + \sqrt{1 - \sigma_1^2})
\;\; \mbox{in $\ol{\Omega}$}.
\end{equation}
\end{lemma}

\begin{proof}
By Lemma~\ref{hb-lemma-L10} we have, for $\epsilon$ sufficiently
small,
\begin{equation}
\label{gsz-L170}
 \sqrt{1 + |Du|^2} \geq \frac{2}{1 + \sigma_2}
 \;\; \mbox{on $\partial \Omega$}.
 \end{equation}
Fix $\lambda > 0$ (sufficiently large) such that
\begin{equation}
\label{gsz-L180} \frac{L}{2 \lambda} \sqrt{\frac{1-\sigma_1}{1+\sigma_1}} \leq
   \ln  \frac{2}{1 + \sigma_2}
\end{equation}
 where $L$ is the diameter of $\Omega$,  and let
\[ u^{\lambda} (x) = \frac{1}{\lambda} u (\lambda x),
 \;\; x \in \Omega^{\lambda} \]
 where $\Omega^{\lambda} = \frac{\Omega}{\lambda}$.
 Then $\kappa [u^{\lambda}] (x) = \kappa [u] (\lambda x)$ in
 $\Omega^{\lambda}$.
  It follow that $u^{\lambda} \in \mathcal{A} (\Omega^{\lambda})$ and
 \begin{equation}
\label{gsz-L150'}
\begin{cases}
\sigma_1 \leq f (\kappa [u^{\lambda}]) \leq \sigma_2,& \;\; \mbox{in $\Omega^{\lambda}$} \\
\;\;\;\;\;\; u^{\lambda} \geq \frac{\epsilon}{\lambda}, & \;\; \mbox{in $\Omega^{\lambda}$} \\
\;\;\;\;\;\; u^{\lambda} = \frac{\epsilon}{\lambda}, & \;\;
\mbox{on $\partial \Omega^{\lambda}$}
\end{cases}
\end{equation}
Applying Lemma~\ref{hb-lemma-L10} to $u^{\lambda}$, we have by
(\ref{gsz-L170}) and (\ref{gsz-L180}),
\[ u^{\lambda} - \frac{\epsilon}{\lambda} \leq
\frac{L}{2 \lambda} \sqrt{\frac{1-\sigma_1}{1+\sigma_1}} \leq \max_{\partial
\Omega^{\lambda}} \ln \sqrt{1 + |D u^{\lambda}|^2}
 \;\; \mbox{in $\Omega^{\lambda}$}\]
 or
 \begin{equation} \label{gsz-L155}
\sup_{\Omega^{\lambda}} e^{u^{\lambda}} \leq \max_{\partial \Omega^{\lambda}}
e^{\frac{\varepsilon}{\lambda}} \sqrt{1+|Du^{\lambda}|^2}.
\end{equation}

By (\ref{gsz-L155}), Lemma~\ref{hb-lemma-L40} and Lemma~\ref{hb-lemma-L10} (part ii, formula (\ref{hb-P110})),
\begin{equation}
 \label{gsz-L190}
 \begin{aligned}
  & \frac{1}{\sqrt{1 + |D u^{\lambda}|^2}}
    \geq e^{(u^{\lambda} - \frac{\epsilon}{\lambda})}
         \min_{\partial \Omega^{\lambda}} \frac1{\sqrt{1 + |Du^{\lambda}|^2}} \\
   & \geq \min_{\partial \Omega^{\lambda}}
          \frac{1}{\sqrt{1 + |Du^{\lambda}|^2}}
   \geq \sigma_1 - C \epsilon (\epsilon + \sqrt{1 - \sigma_1^2}).
 \end{aligned}
\end{equation}
This proves (\ref{gsz-L160}).
\end{proof}

\bigskip

\section{Boundary estimates for second derivatives}
\label{gsz-B}

Let $\Omega$ be a bounded smooth domain in $\bfR^n$.
In this section we establish boundary estimates for second derivatives
of admissible solutions to the Dirichlet problem
\begin{equation}
\label{gsz-I70'}
\left\{ \begin{aligned}
G (D^2 u, Du, u) & = \psi (x, u), & \;\; \mbox{in $\Omega$} \\
u & = \epsilon,  & \;\;  \mbox{on $\partial \Omega$}
\end{aligned} \right.
\end{equation}
where $G$ is defined in (\ref{gs5-F60}) and $\psi \in C^{\infty}
(\ol{\Omega} \times \bfR_+)$. We assume that $\psi$ satisfies the
following conditions:
\begin{equation}
\label{gsz-B20}
 0 < \psi (x, z) \leq \frac{1-\epsilon_1}{z},
 \;\; |D_x \psi (x, z)| + |\psi_z (x, z)| \leq \frac{C}{z^2},
\;\; \forall \; x \in \ol{\Omega_{\delta}}, \; z \in (0, \epsilon_1)
\end{equation}
\begin{equation}
\label{gsz-B30}
\psi (x, z) = \frac{\sigma}{z},
\;\; \forall \; x \in \partial \Omega, \; z \in (0, \epsilon_1).
\end{equation}
where $C$ is a large fixed constant, $\epsilon_1 > 0$ is  a small fixed constant,
$\delta =\frac{\epsilon}{C^2}$ and 
\[ \Omega_{\delta}= \{x\in \Omega: d(x,\partial \Omega)<\delta\}. \]

\begin{remark}
\label{rem1}
 We have in mind $\Psi:=u\psi(x,u)=\sigma+M(u-v(x))$
where $M\in [0,\frac1{\epsilon}]$
and $v=\epsilon~\mbox{on $\partial \Omega$},~|\nabla v|\leq C~\mbox{in $\Omega$}$. We will need
this generality because (see (\ref{gsz-185})) $\tilde{G}_u=G_u-\psi_u$ may be positive in $\Omega$
causing us some trouble when we try to prove Theorem \ref{gsz-th30}.
Note also that  conditions (\ref{gsz-B20}), (\ref{gsz-B30})  imply $\mbox{osc}_{\Omega_{\delta}} \Psi \leq \frac{C}{\epsilon} \delta \leq \frac1C$ which is used
at the end of the proof when we need to appeal to  Lemma~\ref{hb-lemma-L50}.
\end{remark}

\begin{theorem}
\label{gsz-th-B10}
Let $\Omega$ be a bounded domain in $\bfR^n$, $\partial \Omega \in C^3$, and
$u \in C^3 (\bar{\Omega}) \cap \mathcal{A} (\Omega)$ a solution
of problem (\ref{gsz-I70'}). 
Suppose that $f$ satisfies (\ref{gs5-I20})-(\ref{gs5-I45}) and
$\psi$ satisfies (\ref{gsz-B20}) and (\ref{gsz-B30}).
Then, if $\epsilon$ is sufficiently small,
\begin{equation}
\label{hb-I25'}
u|D^2 u| \leq C
\;\;\; \mbox{on $\partial \Omega$}
\end{equation}
where $C$ is independent of
$\epsilon$.
\end{theorem}

The notation of this section follows that of subsection~\ref{gsz-FV}.
We first consider the partial linearized operator of $G$ at $u$:
\[ L = G^{st} \partial_s \partial_t + G^s \partial_s \]
where $G^{st}$ is defined in (\ref{gs5-F70}) and
\begin{equation}
G^s := \frac{\partial G}{\partial u_s}
    = - \frac{u_s}{w^2 u} F^{ij} a_{ij}
      - \frac{2}{w u} 
       F^{ij} a_{ik} \Big(\frac{w u_k \gamma^{sj} + u_j \gamma^{ks}}{1+w}\Big)
      +\frac{2}{w^2 u} F^{ij} u_i \gamma^{sj}
\end{equation}
by the formula (2.21) in \cite{GS04}, where $F^{ij} = F^{ij} (A^{\vv} [u])$
and $a_{ij} = a^{\vv}_{ij} [u]$.
Since $\{F^{ij}\}$ and $\{a_{ij}\}$ are both positive definite and can be
diagonalized simultaneously, we see that
\begin{equation}
\label{gsz-E85}
  F^{ij} a_{ik} \xi_k \xi_j \geq 0, \;\; \forall \; \xi \in \bfR^n.
\end{equation}
Moreover, by the concavity of $f$ we have the following inequality
similar to  Lemma~2.3 in \cite{GS04}
\begin{equation}
\label{gs5-D80}
\sum |G^s| \leq \frac{C}{u} \Big(1 + \sum F^{ii}\Big).
\end{equation}

Since $\gamma^{sj} u_s = u_j/w$,
\begin{equation} \label{gsz-E105}
G^s u_s = \frac1u \Big\{\Big(\frac{1}{w^2} - 1 \Big) F^{ij} a_{ij} 
          - \frac{2}{w^2} F^{ij} a_{ik} u_k u_j + \frac{2}{w^3} F^{ij} u_i u_j \Big\}.
\end{equation}
It follows from (\ref{gs5-F70}) and (\ref{gsz-E105}) that
\begin{equation}
\label{gsz-E110}
L u = \frac{1}{w^2 u} F^{ij} a_{ij} - \frac{1}{w u} \sum F^{ii}
         - \frac{2}{w^2 u} F^{ij} a_{ik} u_k u_j
         + \frac{2}{w^3 u} F^{ij} u_i u_j
\end{equation}

\begin{lemma}
\label{gsz-B10}
Suppose that $f$ satisfies (\ref{gs5-I20}), (\ref{gs5-I30}), (\ref{gsz-I45})
and (\ref{gsz-I40}).
Then
\begin{equation}
\label{gsz-B50}
 L \Big(1 - \frac{\epsilon}{u}\Big)
   \leq - \frac{\epsilon(1 - \frac{u\psi}w)}{2w u^3} \Big(1 + \sum F^{ii}\Big)
\;\; \mbox{in $\Omega$}.
\end{equation}
\end{lemma}

\begin{proof}
By (\ref{gsz-E110}), (\ref{gsz-E85}) and (\ref{gsz-I40}),
\begin{equation}
\label{gsz-E120}
\begin{aligned}
L \frac{1}{u} &   =  - \frac{1}{u^2} L u + \frac{2}{u^3} G^{st} u_s u_t \\
              &   =  - \frac{1}{u^2} L u + \frac{2}{w^3 u^3} F^{ij} u_i u_j \\
              & \geq \frac{1}{w u^3} \Big(\sum F^{ii} - \frac{u\psi}{w}\Big).
\end{aligned}
\end{equation}
Thus (\ref{gsz-B50}) follows from (\ref{gsz-I210}) and (\ref{gsz-B20}).
\end{proof}

\ \begin{lemma}
\label{gsz-lemma-B30} For
 $1 \leq i, j \leq n$,
\begin{equation}
\label{gsz-B40}
 L (x_i u_j - x_j u_i) = (\psi_u - G_u) (x_i u_j - x_j  u_i)
                         + x_i \psi_{x_j} - x_j \psi_{x_i}
\end{equation}
where
\begin{equation}
\label{gsz-B130}
G_u := \frac{\partial G}{\partial u}
     = - \frac{1}{w u^2} \sum_i F^{ii}.
\end{equation}
\end{lemma}

\begin{proof}
For $\theta \in \mathbb{R}$ let
\[ \begin{aligned}
   y_i  & = x_i \cos \theta - x_j \sin \theta \\
   y_j  & = x_i \sin \theta + x_j \cos \theta \\
   y_k  & = x_k, \;\; \forall \; k \neq i, j.
\end{aligned} \]
Differentiate the equation
\[ G(D^2 u(y), D u(y), u(y)) =  \psi (y, u (y)) \]
with respect to $\theta$ and set $\theta = 0$ afterwards. We obtain
\[ \begin{aligned}
  L (x_i u_j - x_j u_i) & + G_u (x_i u_j - x_j  u_i) \\
    & = (L + G_u) \frac{\partial u}{\partial\theta}\Big|_{\theta = 0}
     = \frac{\partial}{\partial\theta} \psi (y, u (y)) \Big|_{\theta = 0}
\end{aligned} \]
which yields (\ref{gsz-B40}).
\end{proof}

\begin{proof}[Proof of Theorem~\ref{gsz-th-B10}]
Consider an arbitrary point on $\partial \Omega$, which we may
assume to be the origin of $\mathbb{R}^n$ and choose the coordinates
so that the positive $x_n$ axis is the interior normal to
$\partial\Omega$ at the origin. There exists a uniform constant
$r > 0$ such that $\partial \Omega \cap B_r (0)$ can be
represented as a graph
\begin{equation}
\label{gsz-B140}
 x_n = \rho(x') = \frac{1}{2} \sum_{\alpha,\beta<n}
     B_{\alpha\beta} x_\alpha x_\beta + O (|x'|^3), \qquad
  x' = (x_1, \dots, x_{n-1}).
\end{equation}
Since $u = \epsilon$ on $\partial \Omega$, we see that
$u (x', \rho (x')) = \epsilon$ and
\begin{equation}
\label{gsz-B150}
 u_ {\alpha\beta} (0) = - u_n \rho_{\alpha\beta} \qquad \alpha,\beta < n.
\end{equation}
Consequently,
\begin{equation}
\label{gsz-B160}
 |u_{\alpha\beta}(0)| \leq C |D u (0)|, \qquad \alpha,\beta < n
\end{equation}
where $C$ depends only on the maximal (Euclidean principal) curvature of
$\partial \Omega$.

Next, following \cite{CNS1} we consider for fixed $\alpha<n$ the operator
\begin{equation}
\label{gsz-B170}
 T = \partial_\alpha + \sum_{\beta <n}
       B_{\alpha\beta} (x_\beta\partial_n- x_n\partial_\beta).
\end{equation}
We have
\begin{equation} \label{gsz-B176}
\begin{aligned}
   |T u| \leq C, \;\; & \mbox{in $\Omega \cap B_{\delta} (0)$} \\
   |T u| \leq C |x|^2,  \;\; & \mbox{on $\partial \Omega \cap B_{\delta}$}
   \end{aligned}
   \end{equation}
since $u = \epsilon$ on $\partial \Omega$.
By Lemma~\ref{gsz-lemma-B30} and (\ref{gsz-B20}), (\ref{gs5-D80}),
\begin{equation}
\label{gsz-B180}
\begin{aligned}
 |L (Tu)| &    = |T G (D^2 u, Du, u) - G_u Tu| \\
          &    = |T \psi (x, u) - G_u Tu| \\
          & \leq \frac{C}{u^2} \Big(1 + \sum F^{ii}\Big).\end{aligned}
\end{equation}
Let
\[ \phi = A \Big(1 - \frac{\epsilon}{u}\Big) + B |x|^2 \pm Tu. \]
By (\ref{gsz-B20}),  (\ref{gs5-D80}), (\ref{gsz-B180}) and Lemma~\ref{gsz-B10},
\begin{equation}
\label{gsz-B185}
L \phi \leq \Big(- \frac{\epsilon_1 \epsilon A}{2 w}+ C u^2 B(u+\delta)+C u\Big)\frac{(1+\sum F^{ii})}{u^3}
\;\;\; \mbox{in $\Omega \cap B_\delta$.}
\end{equation}
We  first choose  $B=\frac{C_1}{\delta^2}$ with $C_1=C$ the constant in (\ref{gsz-B176})
  so that  $\phi \geq 0 ~ \text{on }\partial(\Omega\cap B_\delta)$. Then
choosing $A>>C_1/\epsilon_1$ makes $ L \phi \leq 0 ~ \text{in } \Omega \cap B_\delta$.

By the maximum principle $\phi \geq 0$ in $\Omega \cap B_\delta$.
Since $\phi (0) = 0$, we have $\phi_n (0) \geq 0$ which gives
\begin{equation}
\label{gsz-B200}
 |u_{\alpha n} (0)|\leq \frac{A u_n (0)}{u (0)}.
\end{equation}

Finally to estimate $|u_{nn} (0)|$ we use our hypothesis (\ref{gs5-I45}).
We may assume $[u_{\alpha \beta} (0)]$, $1 \leq \alpha, \beta < n$, to be
diagonal. Note that $u_{\alpha} (0) = 0$ for $\alpha < n$.
We have at $x = 0$
\[ A^{\vv} [u] = \frac{1}{w} \begin{bmatrix}
   1 + u u_{11} &           0  & \dots   & \frac{u u_{1n}}{w} \\
            0   & 1 + u u_{22} & \dots   & \frac{u u_{2n}}{w} \\
   \vdots       &   \vdots     & \ddots  & \vdots \\
\frac{u u_{n1}}{w} & \frac{u u_{n2}}{w} & \dots  & 1 + \frac{u u_{nn}}{w^2}
   \end{bmatrix}. \]

By Lemma 1.2 in \cite{CNS3}, if $\epsilon u_{nn}(0)$ is very
large, the eigenvalues $\lambda_1, \ldots , \lambda_n$ of
$A^{\vv} [u]$ are asymptotically given by
\begin{equation}
\begin{aligned}
 \lambda_{\alpha} = & \frac{1}{w} (1 + \epsilon u_{\alpha \alpha} (0)) + o(1),
\; \alpha < n \\
       \lambda_n  = & \frac{\epsilon u_{nn}(0)}{w^3}
             \Big(1 + O \Big(\frac{1}{\epsilon  u_{nn} (0)}\Big)\Big).
\end{aligned}
\end{equation}
By (\ref{gsz-B160}) and assumptions (\ref{gsz-I40})-(\ref{gs5-I45}),
for all $\epsilon > 0$ sufficiently small,
\[ \epsilon \psi (0, \epsilon) = \frac{1}{w} F (w A^{\vv} [u] (0))
   \geq  \frac{1}{w} \Big(1 + \frac{\varepsilon_0}{2}\Big)    \]
if $\epsilon u_{nn}(0) \geq R$ where $R$ is a uniform constant.
By the hypothesis (\ref{gsz-B30}) and Lemma~\ref{hb-lemma-L50}, however,
\[ \sigma \geq \frac{1}{w} \Big(1 + \frac{\varepsilon_0}{2}\Big)
     \geq (\sigma - C \epsilon)
           \Big(1 + \frac{\varepsilon_0}{2}\Big)
       > \sigma \]
which is a contradiction.
Therefore 
\[ |u_{nn} (0)| \leq  \frac{R}{\epsilon} \]
and the proof is complete.
\end{proof}

\bigskip
\section{Global estimates for second derivatives}
\label{gsz-G}
In this section we prove  a maximum principle for the largest hyperbolic principal curvature $\kappa_{\max} (x)$
of solutions of general curvature equations. For later applications we keep track of how the estimates
depend on the right hand side of (\ref{gsz-I70'}) . We consider
\begin{equation}
\label{gsz-G1}
 M(x) = \frac{\kappa_{\max} (x)}{u^2(x)(\nu^{n+1}(x)-a)}.
\end{equation}

\begin{theorem}\label{gsz-th-G9} Let $u \in C^4 (\bar{\Omega})$ be a  positive solution
of $f(\kappa[u])=\Psi(x,u)$ where $f$ satisfies (\ref{gs5-I20})-(\ref{gsz-I40}), $A^{\vv}[u] \in \mathcal{S}_K,\,\nu^{n+1} \geq 2a$ and
 $\Psi \geq \sigma_0 >0$.
Suppose M achieves its maximum at an interior point $x_0 \in \Omega$. Then
either $\kappa_{\max} (x_0) \leq 16(a+\frac1a)$ or
\begin{equation} \label{gsz-G2}
M^2(x_0) \leq C\frac{\{(|\Psi_x|+|\Psi_u|)^2+(|\Psi_{xx}|+|\Psi_{ux}|+|\Psi_{uu}|)\}(x_0)}{u^2(x_0)}
\end{equation}
where C is a controlled constant.
\end{theorem}

If we assume, for example,  that
\begin{equation}\label{gsz-G3}
\Psi \geq \sigma_0,\, |\Psi_x|+|\Psi_u|  \leq \frac{L_1}{\epsilon},\,
 |\Psi_{xx}|+|\Psi_{ux}|+|\Psi_{uu}|\leq \frac{L_2}{\epsilon^2}\,\,\mbox{in $\Omega$}
\end{equation}
with  $L_1,\,L_2$ independent of $\epsilon$,
we obtain using Theorem \ref{gsz-th-B10}

\begin{theorem}
\label{gsz-th-G10}
Let 
$u \in C^4 (\bar{\Omega})$ be a solution
of problem (\ref{gsz-I70'}) with $A^{\vv}[u] \in \mathcal{S}_K$.
Suppose that $f$ satisfies (\ref{gs5-I20})-(\ref{gsz-I40}) and
$\psi$ satisfies (\ref{gsz-B20}), (\ref{gsz-G3}). 
Then if $\epsilon$ is sufficiently small,
\begin{equation}
\label{gsz-G10}
u|D^2 u| \leq C(1+\max_{\partial \Omega} u|D^2 u|)\frac{u^2}{\epsilon^2}
\;\;\; \mbox{in $\Omega$}
\end{equation}
where $C$  is independent of $\epsilon$.
\end{theorem}

We begin the proof of Theorem \ref{gsz-th-G9} which is long and computational.

Let $\Sigma$ be the graph of $u$.  For $x \in \Omega$ let $\kappa_{\max} (x)$ be the
largest principal curvature of $\Sigma$ at the point $X = (x, u (x)) \in
\Sigma$. We consider
\begin{equation}
\label{gsz-G20}
 M_0 = \max_{x \in \ol{\Omega}} \frac{\kappa_{\max} (x)}{\phi (\eta, u)}
\end{equation}
where $\eta = \ve \cdot \nu$, $\nu$ is the upward (Euclidean) unit normal to
$\Sigma$, and $\phi$ a smooth positive function to be chosen later.
Suppose that $M_0$ is attained at an interior point $x_0 \in \Omega$
and let  $X_0 = (x_0, u (x_0))$.

After a horizontal translation of the origin in $\bfR^{n+1}$,
 we may write $\Sigma$ locally near $X_0$ as a radial graph
\begin{equation}
\label{gsz-F200'}
X = e^v \vz, \;\;\; {\vz} \in \bfS^n_+ \subset \bfR^{n+1}
\end{equation}
with $X_0 = e^{v (\vz_0)} \vz_0$, $\vz_0 \in \bfS^n_+$, such that
$\nu (X_0) = \vz_0$.

 In the rest of this section we shall follow the
notation in subsection~\ref{gsz-FR} and rewrite the equation in
(\ref{gsz-I70'}) in the form
\begin{equation}
\label{gsz-F100'}
F (A^{\vs} [v]) = \Psi \equiv u \psi
\end{equation}
where $A^{\vs} [v]$ is given in (\ref{gsz-F20'});
henceforth we write $A [v] = A^{\vs} [v]$ and $a_{ij} = a^{\vs}_{ij} [v]$.

 We choose an orthnormal local frame $\tau_1, \ldots, \tau_n$ around
$\vz_0$ on $\bfS^n_+$ such that $v_{ij} (\vz_0)$ is diagonal. Since
$\nu (X_0) = \vz_0$,  $\nabla v (\vz_0) = 0$ and, by
(\ref{gsz-F20'}), at $\vz_0$,
\begin{equation}
\label{gsz-G55}
 a_{ij} = y v_{ij} = \kappa_i
\delta_{ij}
\end{equation}
 where $\kappa_1, \ldots, \kappa_n$ are the principal
curvatures of $\Sigma$ at $X_0$. We may assume
\begin{equation}
\label{gsz-G60}
 \kappa_1 = \kappa_{\max} (X_0).
\end{equation}

The function $\frac{a_{11}}{\phi}$, which is defined locally near
$\vz_0$, then achieves its maximum at $\vz_0$ where, therefore
\begin{equation}
\label{gsz-G70}
 \Big(\frac{a_{11}}{\phi}\Big)_i = 0, \;\; 1 \leq i \leq n
\end{equation}
 and
\begin{equation}
\label{gsz-G80}
F^{ii} \Big(\frac{a_{11}}{\phi}\Big)_{ii} \\
= \frac{1}{\phi} F^{ii} a_{11, ii}
       - \frac{\kappa_1}{\phi^2}  F^{ii} \phi_{ii} \leq 0.
\end{equation}

\begin{proposition} At  $\vz_0$,
\label{gsz-prop5.1}
\begin{equation}\label{gsz-G81}
 \begin{aligned}
  y^2 \phi F^{ii} a_{11, ii} & \, - y^2 \kappa_1  F^{ii} \phi_{ii} \\
      = & \, y^2 \phi F^{ii} a_{ii, 11} + y (\phi_{\eta}\kappa_1 - \phi )F^{ii} y_j a_{ii,j}
             - 2 y\phi F^{ii}  y_1 a_{ii,1} \\
        & \, + (y \phi_{\eta} - \phi) \kappa_1 \sum f_i \kappa_i^2
   - \phi_{\eta \eta} \kappa_1 \sum f_i (y - \kappa_i)^2 y_i^2 \\
  & \, + (\phi (\kappa_1^2 + 2 y_1^2 + 1) - y^2 e^v \phi_u \kappa_1
       - (1 + y^2) \phi_{\eta} \kappa_1) \sum f_i \kappa_i \\
  & \, + (y \phi_{\eta} + y e^v \phi_u
           - (1 + 2 y_1^2) \phi) \kappa_1 \sum f_i \\
  & \, + (2 y \phi_{\eta} + 2 y e^v \phi_u - y^2 e^{2v} \phi_{uu}
        - 2 y^2 e^v \phi_{u \eta} - 2 \phi) \kappa_1 \sum f_i y_i^2  \\
  & \, + 2 ( \phi - \phi_{\eta} \kappa_1 + y e^v \phi_{u \eta} \kappa_1)
         \sum f_i \kappa_i y_i^2.
          \end{aligned}
\end{equation}
\end{proposition}
\begin{proof}
In what follows all calculations are evaluated at $\vz_0$.
 Since $\nabla v = 0$, we have 
\begin{equation}
\label{gsz-G90}
 w = 1, \;\; w_i = 0, \;\; w_{ij} = v_{ki} v_{kj}
\end{equation}
(recall $w = \sqrt{1 + |\nabla v|^2}$). Straightforward calculations
show that
\begin{equation}
\label{gsz-G120}
 a_{ij, k} = y v_{ijk} + y_k v_{ij} - (\ve \cdot
\nabla v)_k \delta_{ij}
\end{equation}
\begin{equation}
\label{gsz-G130}
 a_{kk, ii} =  y v_{kkii} + 2 y_i v_{kki} + y_{ii} v_{kk}
                - (\ve \cdot \nabla v)_{ii} - y v_{kk} v_{ii}^2
                - 2 y v_{kk}^3 \delta_{ki}.
\end{equation}
Therefore,
\begin{equation}
\label{gsz-G135}
\begin{aligned}
 a_{11, ii} - a_{ii, 11} = \, &
    y (v_{11ii} - v_{ii11} - v_{11} v_{ii}^2 + v_{ii} v_{11}^2) \\
    & + y_{ii} v_{11} - y_{11} v_{ii} + 2 (y_i v_{11i} - y_1 v_{ii1}) \\
    & + (\ve \cdot \nabla v)_{11} - (\ve \cdot \nabla v)_{ii}.
   \end{aligned}
\end{equation}
We recall the following formulas
\begin{equation}
\label{gsz-G100}
 v_{ijk} = v_{ikj} = v_{kij}
\end{equation}
\begin{equation}
\label{gsz-G110}
 v_{kkii} = v_{iikk} + 2 (v_{kk} - v_{ii})
\end{equation}
(where we have used the fact that $\nabla v = 0$) and, from
\cite{GS00},
\begin{equation}
\label{gsz-G155}
 \sum y_j^2 = 1 - y^2
\end{equation}
\begin{equation}
\label{gsz-G140}
 y_{ij} = - y \delta_{ij}
\end{equation}
\begin{equation}
\label{gsz-G150}
  (\ve \cdot \nabla v)_i = y_i v_{ii}
\end{equation}
\begin{equation}
\label{gsz-G160}
 \begin{aligned}
 (\ve \cdot  \nabla v)_{ij}
 = \, & \ve \cdot \tau_k v_{kij} - 2 y v_{ij} - \ve \cdot \tau_j v_i \\
 = \, & y_k v_{ijk} - 2 y v_{ij} - y_j v_i.
  \end{aligned}
\end{equation}
By (\ref{gsz-G120}) and (\ref{gsz-G55}),
\begin{equation}
\label{gsz-G120'}
 v_{iij} = \frac{a_{ii,j}}{y} -  \frac{y_j \kappa_i}{y^2}
           + \frac{y_j \kappa_j}{y^2}
\end{equation}
\begin{equation}
\label{gsz-G170}
 (\ve \cdot  \nabla v)_{ii}
 = \frac{y_j a_{ii, j}}{y} - \Big(1 + \frac{1}{y^2}\Big) \kappa_i
    + \frac{1}{y^2} \sum_j \kappa_j y_j^2.
\end{equation}
Plug these formulas into (\ref{gsz-G135}) and note that
 $F^{ij} = f_i \delta_{ij}$. We obtain
\begin{equation}
\label{gsz-G180}
 \begin{aligned}
y^2 F^{ii} a_{11, ii}
 = \, & y^2  F^{ii} a_{ii, 11}
        - y F^{ii} (y_j a_{ii,j} + 2 y_1 a_{ii,1}) \\
      & + y F^{ii} (y_j a_{11,j} + 2 y_i a_{11,i})
        - \kappa_1 \sum f_i \kappa_i^2 \\
      & + (\kappa_1^2 + 2 y_1^2 + 1) \sum f_i \kappa_i
        - (1 + 2 y_1^2) \kappa_1 \sum f_i \\
      & + 2 \sum f_i \kappa_i y_i^2 - 2 \kappa_1 \sum f_i y_i^2.
  \end{aligned}
\end{equation}

Next, recall that $u = y e^v$ and
$\eta := \ve \cdot \nu = \frac{y - \ve \cdot \nabla v}{w}$.
 At $\vz_0$ we have $\eta = y$,
\begin{equation}
\label{gsz-G190}
 \eta_i = y_i - (\ve \cdot \nabla v)_i = y_i (1 - v_{ii})
\end{equation}
\begin{equation}
\label{gsz-G200}
 \begin{aligned}
 \eta_{ii}
  = & \; y_{ii} - (\ve \cdot \nabla v)_{ii} - y v_{ii}^2 \\
  = & \; - y - \frac{\kappa_i^2}{y} + \Big(1 + \frac{1}{y^2}\Big)
   \kappa_i \\
    & - \frac{1}{y^2} \sum_j (y y_j a_{ii,j} + \kappa_j y_j^2)
  \end{aligned}
\end{equation}
and
\begin{equation}
\label{gsz-G210}
 u_i =  e^v y_i, \;\; u_{ii} =  e^v (\kappa_i - y).
\end{equation}
We have
\begin{equation}
\label{gsz-G220}
 \begin{aligned}
 y^2 F^{ii} \phi_{ii}
 = & \,  y^2 F^{ii} (\phi_{\eta} \eta_{ii} + \phi_{\eta \eta} \eta_i^2
  + 2 \phi_{u \eta} u_i \eta_i + \phi_{uu} u_i^2 + \phi_u u_{ii}) \\
 = & \, -  y \phi_{\eta} \sum f_i \kappa_i^2
        + (y^2 e^v \phi_u + (1 + y^{2}) \phi_{\eta}) \sum f_i \kappa_i \\
   & \, - ( y^3 \phi_{\eta} + y^3 e^v \phi_u +  \kappa_j y_j^2 \phi_{\eta})
          \sum f_i
        + \phi_{\eta \eta} \sum f_i y_i^2 (y - \kappa_i)^2 \\
   & \, + (y^2 e^{2v} \phi_{uu} + 2 y^2 e^v \phi_{u \eta}) \sum f_i y_i^2
        - 2 y e^v \phi_{u \eta} \sum f_i \kappa_i y_i^2\\
   &\,-y\phi_{\eta}\sum yy_j F^{ii}a_{ii,j}.
\end{aligned}
\end{equation}
By (\ref{gsz-G70}),
\begin{equation}
\label{gsz-G230}
 a_{11,i} \phi
   = \kappa_1 \phi_i
   = \kappa_1 (\phi_{\eta} \eta_i + e^v \phi_u y_i)
   = \kappa_1 \phi_{\eta} \Big(1 - \frac{\kappa_i}{y}\Big) y_i
    + e^v \phi_u \kappa_1 y_i.
\end{equation}
Using (\ref{gsz-G230}) we have
\begin{equation}\label{gsz-G231}
\begin{aligned}
y\phi F^{ii}(y_j & \, a_{11,j}+2y_i a_{11,i}) = \\
& \,2(y\phi_{\eta}+ye^v \phi_u)\kappa_1 \sum f_i y_i^2
  -2 \phi_{\eta}\kappa_1 \sum f_i \kappa_i y_i^2 \\
& + \Big(y (1-y^2)\phi_{\eta}+y(1-y^2)e^v \phi_u
-\phi_{\eta}\sum \kappa_j y_j^2\Big) \kappa_1 \sum f_i.
\end{aligned}
\end{equation}
Combining (\ref{gsz-G80}),  (\ref{gsz-G180}),
(\ref{gsz-G220}) and (\ref{gsz-G231}), we obtain (\ref{gsz-G81}).
\end{proof}

\begin{lemma}
\label{gsz-lem5.1}
\begin{equation}\label{gsz-G241}
\begin{aligned}
y^2 F^{ii} a_{ii,11} & \, -2yF^{ii}y_1 a_{ii,1}
+ y \Big(\frac{\phi_{\eta}}{\phi}\kappa_1 - 1\Big) F^{ii}y_j a_{ii,j}
\geq -y^2F^{ij,kl}a_{ij,1}a_{kl,1} \\
& -C\{u\kappa_1 (|\Psi_x|+|\Psi_u|) + u^2 (|\Psi_{xx}|+|\Psi_{ux}|+|\Psi_{uu}|)\}
\end{aligned}
\end{equation}
where C depends on an upper bound for $|\frac{\phi_{\eta}}{\phi}|$.
\end{lemma}
\begin{proof}
Since $X=e^v \vz,~x=e^v(\vz-y\ve)~\mbox{and}~u=e^v y$. Hence,
\begin{equation}
\label{gsz-G242}
\begin{aligned}
&\nabla_{\tau_j} x=e^v(\tau_j-y_j\ve) + v_j x \\
&\nabla_{\tau_j}\Psi(x,u)=e^v(\Psi_x \cdot (\tau_j-y_j \ve)+\Psi_u (yv_j+y_j))+(x \cdot \Psi_x)v_j \\
&\nabla_{\tau_1 \tau_1}\Psi(x,u)=e^{2v}(D_x^2 \Psi(\tau_1-y_1 \ve)(\tau_1-y_1 \ve)-
2y_1 \Psi_{ux}\cdot (\tau_1-y_1 \ve)+\Psi_{uu}y_1^2)\\
&+ e^v(\Psi_x \cdot (\nabla_{\tau_1}\tau_1-y_{11}\ve)+\Psi_u(\kappa_1+y_{11}))+(x \cdot \Psi_x)v_{11}.
\end{aligned}
\end{equation}
Using (\ref{gsz-G242}) and differentiating equation~(\ref{gsz-F100'}) twice gives
\begin{equation}
\label{gsz-G245}
\begin{aligned}
&yF^{ii}y_j a_{ii,j}=u(\Psi_x\cdot (y_j \tau_j-(1-y^2)\ve)+\Psi_u (1-y^2))\\
&yF^{ii}y_1 a_{ii,1}=u(\Psi_x \cdot (y_1 \tau_1-y_1^2 \ve) +\Psi_u y_1^2)\\
 &y^2F^{ii} a_{ii,11} =u^2(D_x^2 \Psi(\tau_1-y_1 \ve)\cdot (\tau_1-y_1 \ve)
 +2y_1\Psi_{ux} \cdot (\tau_1 - y_1 \ve)+\Psi_{uu}y_1^2)\\
  &+yu(\Psi_x \cdot (\nabla_{\tau_1}\tau_1+ y \ve)+\Psi_u(\kappa_1-y))+(x \cdot \Psi_x)y\kappa_1
  - y^2 F^{ij,kl} a_{ij,1} a_{kl,1}.
\end{aligned}
\end{equation}
Formula (\ref{gsz-G241}) follows immediately from (\ref{gsz-G245}).
\end{proof}

We now make the choice $\phi (\eta, u) = (\eta - a) u^2$ where $0
< a \leq \eta/2$. We have
\[ (y - a) \phi_{\eta} = \phi, \;\; \phi_{\eta \eta} = 0,    \;\; u \phi_u = u^2 \phi_{uu} = 2\phi,
\;\; u(y-a)\phi_{u\eta}=2\phi.\]
By  Proposition \ref{gsz-prop5.1} and Lemma \ref{gsz-lem5.1}, for $\kappa_1 \geq 16(a+\frac1a)$
\begin{equation}
\label{gsz-G250}
 \begin{aligned}
-&y^2 (y - a) F^{ij,kl} a_{ij,1}a_{kl,1}
     + a \kappa_1 \sum f_i \kappa_i^2
     +\frac{a}2 \sigma_0  \kappa_1^2\\
        +& (a+2y^2(y-a)) \kappa_1 \sum f_i
     +2 (\kappa_1 + y - a) \sum f_i (\kappa_i-y) y_i^2
        + 2 y (y - a) \sum f_i y_i^2\\
        & \leq C\{u\kappa_1 (|\Psi_x|+|\Psi_u|) + u^2(|\Psi_{xx}|+|\Psi_{ux}|
        +|\Psi_{uu}|)\}.
  \end{aligned}
\end{equation}

Let $0 < \theta < 1$  (to be chosen in a moment) and set
\[  \begin{aligned}
I & = \{i: \kappa_i \leq - \theta \kappa_1\}, \\
J & = \{i: - \theta \kappa_1 < \kappa_i \leq y, \; f_i < \theta^{-1} f_1\}, \\
K & = \{i: - \theta \kappa_1 < \kappa_i \leq y, \; f_i \geq \theta^{-1} f_1\}, \\
L & = \{i: \kappa_i > y\}.
  \end{aligned} \]
Note that for $i\in L$, all the terms on the left hand side of (\ref{gsz-G250})are nonnegative.

We have (provided that $\theta \kappa_1 \geq 1$),
\begin{equation}
\label{gsz-G260}
\begin{aligned}
 \sum_{i \in I} f_i \kappa_i^2
 \geq & \, \frac{1}{2} \sum_{i \in I} f_i (\kappa_i^2 + \theta^2 \kappa_1^2) \\
 \geq & \, \frac{\theta \kappa_1}{2}
           \sum_{i \in I} f_i (|\kappa_i| + \theta \kappa_1) \\
 \geq & \, \frac{\theta \kappa_1}{2}
           \sum_{i \in I} f_i (|\kappa_i| + y) y_i^2,
  \end{aligned}
\end{equation}
 and
\begin{equation}
\label{gsz-G265}
 \sum_{i \in J} f_i (\kappa_i - y) y_i^2 \geq - 2\kappa_1   f_1.
\end{equation}
According to Andrews~\cite{Andrews94}
and Gerhardt~\cite{Gerhardt96} (see also \cite{Urbas02}, Lemma 3.1 and
\cite{SUW04}),
\begin{equation}
\label{gsz-G270}
 - F^{ij,kl} a_{ij,1} a_{kl,1}
 \geq  \sum_{i \neq j} \frac{f_i - f_j}{\kappa_j - \kappa_i} a_{ij,1}^2
 \geq 2 \sum_{i=2}^n \frac{f_i - f_1}{\kappa_1 - \kappa_i} a_{i1,1}^2.
\end{equation}
By (\ref{gsz-G120}) and (\ref{gsz-G230}),
\[  \begin{aligned}
\label{gsz-G280}
 y a_{i1,1}
 = y a_{11,i} + \kappa_i y_i - \kappa_1 y_i
 = \Big(\kappa_i + \kappa_1  + \frac{y - \kappa_i}{y - a} \kappa_1\Big)
 y_i.
  \end{aligned} \]
Therefore,
\[ y^2 a_{i1,1}^2 \geq 
 \frac{2 (1 - \theta) \kappa_1^2}{y -a}  (y -  \kappa_i) y_i^2,
 \;\; \; \forall \; i \in K. \]
Note that
\begin{equation}
\label{gsz-G290}
 \frac{f_i - f_1}{\kappa_1 - \kappa_i}
\geq \frac{f_i - \theta f_i}{\kappa_1 + \theta \kappa_1}
=  \frac{(1 - \theta) f_i}{(1 + \theta) \kappa_1},
\;\;\forall \; i \in K.
\end{equation}
It follows that
\begin{equation}
\label{gsz-G300}
 -y^2 (y - a) F^{ij,kl} a_{ij,1} a_{kl,1}   \geq  \frac{4 (1 - \theta)^2 }{ 1 + \theta}\kappa_1
                   \sum_{i \in K} f_i (\kappa_i - y) y_i^2.
\end{equation}
We now fix $\theta$ such that
\[ \frac{4 (1 - \theta)^2}{ 1 + \theta} \geq 2 + \theta. \]
For example, we can choose $\theta=\frac{1}{6}$.
From (\ref{gsz-G260}), (\ref{gsz-G265}) and (\ref{gsz-G300})
we obtain
\begin{equation}
\label{gsz-G310}
y^2 (y - a)  F^{ij,kl} a_{ij,1} a_{kl,1} + a \kappa_1 \sum f_i \kappa_i^2
       + 2 (\kappa_1 + y - a) \sum f_i (\kappa_i -y) y_i^2 \geq 0
\end{equation}
provided that $\kappa_1\geq 16(a+\frac1a)$.  Consequently,
\begin{equation}
\label{gsz-G320}
     \frac{a\sigma_0}2 \kappa_1^2
      \leq C\{(|\Psi_x|+|\Psi_u|)u\kappa_1+u^2(|\Psi_{xx}|+|\Psi_{ux}|
        +|\Psi_{uu}|)\}
\end{equation}
Formula (\ref{gsz-G2}) follows easily from (\ref{gsz-G320}) completing the proof of Theorem \ref{gsz-th-G9}.

\bigskip

\section{Existence: Proof of Theorems~\ref{gsz-th30} and \ref{gsz-th20}}
\label{gsz-P}

In order to prove Theorem \ref{gsz-th30}  we will construct a monotone  sequence $\{u_k \}$
of admissible functions satisfying (\ref{gsz-I20}) in $\Omega$
starting from $u_0 \equiv \epsilon$. Having found $u_0\leq u_1 \leq \ldots \leq u_k,~ u=u_{k+1}$ is a solution of the Dirichlet problem
\begin{equation}
\label{gsz-G350}
\begin{aligned}
G(D^2 u, Du, u)&=\frac1{u} \Big(\sigma + \frac1{\epsilon}(u-u_k)\Big) \equiv \psi(x,u)~\mbox{in}~\Omega. \\
&u=\epsilon~\mbox{on}~\partial \Omega
\end{aligned}
\end{equation}

In order to solve (\ref{gsz-G350}) we use a continuity method for $u=u^t, ~0\leq t \leq1$:
\begin{equation}
\label{gsz-G355}
\begin{aligned}
G(D^2 u, Du, u)&=
\frac1{u} \Big(\sigma +\frac1{\epsilon}(u-(tu_k+(1-t)u_{k-1}))\Big) ~\mbox{in}~\Omega,~ ~~\,k\geq 1 \\
G(D^2 u, Du, u)&=\frac1{u} \Big(t(\sigma-1)+\frac1{\epsilon}u\Big)~\mbox{in}~\Omega,~ ~~k=0 \\
u&=\epsilon~\mbox{on}~\partial \Omega
\end{aligned}
\end{equation}
where $u\in \mathcal{A}_k=\{u\geq u_k ~\mbox{ and $u$ admissible}\}$ and $u^0=u_k$. Since u
is admissible we have from section 3 that $|u|_{C^1{\Omega}}\leq C$ for a uniform constant C.
Now according to (\ref{gsz-185}),
\begin{equation}
\label{gsz-380}
\begin{aligned}
&G_u-\psi_u \leq \frac1{u^2} \Big(\sigma-\frac1w-\frac1{\epsilon}(tu_k+(1-t)u_{k-1})\Big)
\leq -\frac{1-\sigma+\frac1C}{u^2}, \,\,~\mbox{ $k \geq 1$}\\
&G_u-\psi_u \leq \frac1{u^2} \Big(t(\sigma-1)-\frac1w\Big)\leq
-\frac1{Cu^2},\,\,~~\mbox{ $k=0$.}
\end{aligned}
\end{equation}
Hence for $\Omega \in C^{2+\alpha}$, the linearized operator for $u^t$ is invertible  and the set of t for which (\ref{gsz-G355}) is solvable  is open. In particular,  (\ref{gsz-G355}) is solvable for $0 \leq t \leq 2t_0$.
Using standard regularity theory for concave fully nonlinear operators, to show the closedness of this set, it suffice to show $|u|_{C^2(\Omega)}\leq C$ for a uniform constant C for $t_0\leq t \leq 1$.
Observe that  $\Psi(x,u)=u\psi(x,u)=\sigma +\frac1{\epsilon}(u-(tu_k+(1-t)u_{k-1}))$ satisfies the conditions (\ref{gsz-B20}), (\ref{gsz-G3}) of Theorem \ref{gsz-th-G10}. Hence we obtain an estimate
\[ \sup_{\Omega}|D^2 u| \leq \frac{C_k}{\epsilon^3} \]
where $C_k$ depends on $k$ but is independent of $t$. Therefore (\ref{gsz-G355}) is solvable for all
$0 \leq t \leq 1$ and so we have found  a monotone increasing sequence of solutions to (\ref{gsz-G350}).\\

It remains to show that the  sequence $\{u_k\}$ converges to a solution of (\ref{gsz-I70}). For this we need
second derivative estimates independent of $k$. Define
  \[  M_{k}(x) =  \frac{\kappa_{\max} (x)}{u_k^2(x)(\nu_k^{n+1}(x)-a)}~.\]
 If  $M_{k}(x) $ achieves its maximum on $\partial \Omega$, then according to Theorem \ref{gsz-th-B10},
 $M_{k}(x) \leq \frac{C}{\epsilon^2}$ where $C$ is independent of $k$ and $\epsilon$ (see Remark \ref{rem1}). Otherwise
applying Theorem~\ref{gsz-th-G9} with $\Psi(x,u)=\sigma+\frac1{\epsilon}(u-u_k)$ we obtain
\begin{equation}
 \label{gsz-G400}
M_{k+1}^2 \leq \frac{C}{\epsilon^4}+\frac{C}{\epsilon^2}M_k \leq
\frac{C}{\epsilon^4}+\frac12 M_k^2
\end{equation}
where $C$ is independent of $k$ and $\epsilon$. Iterating (\ref{gsz-G400}) gives
\[M_k^2 \leq \frac{2C}{\epsilon^4}+\frac12 M_1^2\leq \frac{C}{\epsilon^4}.\]
It follows that the sequence $u_k$ converges uniformly in $C^{2+\alpha}(\ol{\Omega})$
and the proof of Theorem \ref{gsz-th30} is complete. \\

To finish the proof of Theorem \ref{gsz-th20} we need to show that for $\sigma^2>\frac18$, we can obtain an estimate for $\sup_{\Omega}\kappa_{\max}$ which is independent of $\epsilon$ as $\epsilon$
tends to zero.

As in Section 5 we define
\[ M_0 = \max_{x \in \ol{\Omega}} \frac{\kappa_{\max} (x)}{\phi (\eta, u)}~.\]
We now choose $\phi = \eta -a$ ,  where $\inf \eta >a$. If $M_0$ is achieved on $\partial \Omega$, then
we obtain a uniform bound by Theorem \ref{gsz-th-G9}. Otherwise at an interior maximum,
Proposition \ref{gsz-prop5.1} and Lemma \ref{gsz-lem5.1} gives
\begin{equation}
\label{gsz-G500}
 \sigma (y-a) \kappa_1^2  + (a - 2 (1-y^2) (y - a))  \kappa_1 \sum f_i
  \leq 4\sigma\kappa_1
\end{equation}
where we have dropped some positive terms from the left hand side of (\ref{gsz-G500})
 and used $\sum f_i \kappa_i y_i^2 \leq \sigma$.
 From (\ref{gsz-G500}) we see that we must find the minimum of the function
 \begin{equation} 
 \label{gsz-G550}
 \gamma(y)=a - 2 (1-y^2) (y - a)=2y^3-2ay^2-2y+3a ~\,\, \mbox{on $[a, 1]$}.
\end{equation}
We have
\begin{equation}
 \label{gsz-G550'}
 \begin{aligned}
 \,&\gamma^{\prime}(y)=2(3y^2-2ay-1)\\
 \,&\gamma''(y)=4(3y-a)
 \end{aligned}
 \end{equation}
The unique critical point of $\gamma(y)$ in $(a, 1)$ is $y^*=\frac{a+\sqrt{a^2+3}}3$ and some computation
shows that
\[\gamma(y^*)=\frac73 a -\frac4{27}a^3-\frac4{27}(a^2+3)^{\frac32}.\]
It is also not difficult to see that $\gamma(y^*)<a=\gamma(a)=\gamma(1)$. 

We claim $\gamma(y^*)>0$ if $a^2> \frac18$.
This is equivalent to showing
\[ 4(a^2+3)^{\frac32}<a(63-4a^2)\]
which after squaring both sides is in turn equivalent to
\[a^4-\frac{131}{24} a^2+\frac23 =(a^2-\frac18)(a^2-\frac{16}3)<0~.\]
Thus our claim follows.

Now suppose $2\varepsilon_0=\sigma^2-\frac18>0$ and set $a^2=\frac18+\varepsilon_0$. Then
\[\sigma-a=\frac{\varepsilon_0}{\sigma+a}>\frac{\varepsilon_0}{2\sigma}~.\]
According to Lemma~\ref{hb-lemma-L50} (see formula (\ref{gsz-L160})), $\eta \geq \sigma-C \epsilon$
for a uniform constant $C$ if $\epsilon$ is sufficiently small.  Hence if $C \epsilon \leq \frac{\varepsilon_0}{4\sigma}$,
\[ \eta-a \geq (\sigma-a)-C\varepsilon \geq \frac{\varepsilon_0}{2\sigma}-C\varepsilon>
\frac{\varepsilon_0}{4\sigma}~.\]
Returning to formula (\ref{gsz-G500}) we find
\[\frac{\varepsilon_0}4 \kappa_1^2 \leq 4\sigma\kappa_1 \]
or
\[ \kappa_1 \leq \frac{16\sigma}{\varepsilon_0}=\frac{16\sigma}{\frac18-\sigma^2}~.\]
The proof of Theorem \ref{gsz-th20} is complete.

\end{document}